\documentclass[a4paper,11pt]{amsproc} 
\usepackage{amssymb}
\usepackage{amsthm}
\usepackage{epsf} 
\usepackage{bm}
\usepackage[dvipsnames]{xcolor}
\usepackage[T1]{fontenc}   
\usepackage{tikz}         
\makeindex
\newtheorem{thm}{Theorem}[section]

\def\c1{\operatorname{c_1}}
\def\c2{\operatorname{c_2}}

\def\S{{\mathbb S}}

\def\D{{\mathcal D}}
\def\B{{\mathcal B}}
\def\R{{\mathbb R}}

\def\F{{\mathcal F}}

\def\V{{\mathcal V}}

\def\+{\oplus}                   
\def\*{\otimes}                  

\def\V{\mathcal{V}}

\def\P{\mathcal{P}}
\def\a{\alpha}
\begin{document}
\title[]{A generalization of Gerzon's bound on spherical $s$-distance sets}

\author[Datta]{Mrinmoy Datta}
\address{Department of Mathematics, 
 IIT Hyderabad  \newline \indent 
}
\email{mrinmoy.datta@math.iith.ac.in} 
\thanks{The first named author is partially supported by a Seed Grant from Indian Institute of Technology Hyderabad, whereas the second named author is partially supported by a doctoral scholarship from Council of Scientific and Industrial Research, India.}

\author[Manna]{Subrata Manna}
\address{Department of Mathematics, 
 IIT Hyderabad  \newline \indent 
}
\email{subrata147314@gmail.com} 
\thanks{}


\begin{abstract}
We Use the method of linearly independent polynomials to derive an upper bound for the cardinality of a spherical $s$-distance set $\F$ where the sum of distinct inner products of any two elements from $\F$ is zero. Our result generalizes the well-known Gerzon's bound for the cardinality of an equiangular spherical set to a significantly broader class of spherical $s$-distance sets.  
\end{abstract}

\maketitle

\section{Introduction}
Let $n$ be a positive integer and we endow $\R^n$ with the Euclidean norm $|| . ||$. A finite subset $\F \subset \R^n$ is said to be an $s$-\textit{distance set} if the cardinality of $$\D(\F) = \left\{||x - y|| \mid x, y \in \F, x \neq y\right\}$$ is equal to $s$. An $s$-distance set $\F$ is called a \textit{spherical $s$-distance set} if $\F \subset \S^{n-1}$, where $\S^{n-1}$ denotes the unit sphere with center at origin in $\R^n$. The problems concerning the determination of the cardinality of a maximal $s$-distance set and a maximal spherical $s$-distance set in $\R^n$ have been well studied. The first known upper bound for the cardinality of a spherical $s$-distance set is due to Delsarte, Goethals and Seidel, who proved \cite{DGS} that such a set in $\R^n$ contains no more than ${n + s - 1 \choose s} + {n + s - 2 \choose s-1}$ points.

Several other bounds, albeit in special cases, are found in literature. A spherical $2$-distance set is called \textit{equiangular} if there exists $\a \in (0, 1]$ such that $\D(\F) = \{ \a, -\a\}$. Among various well-known results towards upper bounds for cardinalities of equiangular sets in $\R^n$, we are interested in generalizing a bound, popularly known as Gerzon's bound \cite[Theorem 3.5]{LS}, that tells us that an equiangular set in $\R^n$ contains at most ${n + 1 \choose 2}$ many points. 

Let $\F \subset \S^{n-1}$ be a spherical $s$-distance set. It follows readily that there is a one-to-one correspondence between the set 
$\D(\F)$ and $$\P(\F) = \left\{\langle x, y \rangle \mid x, y \in \F, x \neq y\right\},$$ where $\langle . , . \rangle$ denotes the standard inner product on $\R^n$. The Gerzon's bound for the cardinality of a spherical $2$-distance set was generalized by Musin. He proved \cite[Theorem 1]{M} that if $\F$ is a two distance set with $\P(\F) = \{t_1, t_2\}$ satisfying $t_1 + t_2 \ge 0$, then $|\F| \le {n+1 \choose 2}$.
This result was further extended by Barg and Musin to a far more general setting. They proved \cite[Theorem 15]{BM} that if $s$ is even and $t_1 + \dots + t_s \ge 0$, then the cardinality of a spherical $s$-distance set $\F \subset \R^n$ with $\P(\F) = \{t_1, \dots, t_s\}$ is bounded above by the quantity $BM (s, n)$, where 
\begin{equation}\label{bound:bm}
BM (s, n) = {n + s - 3 \choose s - 2} + {n+s - 4 \choose s - 3}+ \frac{n + 2s - 2}{s}{n + s - 3 \choose s - 1}.
\end{equation}
In a recent work \cite{H}, Heged\"us proved the following result: If $s$ is even, say $s = 2 \ell$ and $\F$ is a spherical $s$-distance set with $\P (\F) = \{\a_1, \dots, \a_{\ell}, -\a_1, \dots, - \a_{\ell}\}$ for some $\a_1, \dots, \a_{\ell} \in (0, 1]$, then $|\F| \le {n + s - 1 \choose s}$. Note that this result also generalizes the Gerzon's bound to a significantly broader setting, allowing $s$ to be any positive even integer, however,  restricting attention to a very special subcase of $t_1 + \dots + t_s = 0$. 

In this work, we are interested in determining a bound for the cardinality of a spherical $s$-distance set $\F$ in $\R^n$ satisfying the condition that $\sum_{t \in \P(F)} t = 0$. Thus as compared to the setting corresponding to the bound obtained by Barg and Musin, we relax the constraint on $s$ to be an even integer, but impose a condition on the sum of the $s$ distinct inner products that is akin to the setting of Gerzon's bound. Furthermore, we show at the end of this article that our bound (cf. Theorem \ref{main}) coincides with Barg and Musin's bound in the case when $s$ is even and $\sum_{t \in \P(\F)} t  = 0$. We shall make use of the method of linearly independent polynomials to prove our theorem. We remark that this method was previously used in  \cite{A, B, M} among various other articles.

This paper is organized as follows: In Section \ref{prel}, we introduce the notations and recall the preliminaries that will be used in proving the main theorem. We state and prove our main theorem in Section \ref{three} that ends with a comparison of our bound to that of Barg and Musin. 

\section{Preliminaries and notations}\label{prel}
In this section, we recall some basic facts that will be used in the proof of our main theorem. None of the results in this section are original in nature, but they have been collected for the ease of reference. 

\subsection{Polynomial functions on the unit circle}
Let $\S^{n-1}$ denote the unit sphere centered at origin in $\R^n$. This is given by the vanishing set of the polynomial $x_1^2 + \dots + x_n^2 - 1$ in $\R^n$. A function $\phi : \S^{n-1} \to \R$ is said to be a \textit{polynomial function (of degree $d$)} if there exists a polynomial $F \in \R[x_1, \dots, x_n]$ (with $\deg F = d$) such that $F (a_1, \dots, a_n) = \phi (a_1, \dots, a_n)$ for all $(a_1, \dots, a_n) \in \S^{n-1}$. It is easy to see that the set of all polynomial functions on $\S^{n-1}$ forms a ring that is isomorphic to the coordinate ring $\Gamma (\S^{n-1}) = \R[x_1, \dots, x_n]/I$ of $\S^{n-1}$, where $I$ denotes the ideal in $\R[x_1, \dots, x_n]$ consisting of all polynomials in $\R[x_1, \dots, x_n]$ that vanish at all the points in $\S^{n-1}$. It is not difficult to show that the set of monomials 
$$\B = \left\{x_1^{i_1}\cdots x_n^{i_n} \mid i_1 \le 1\right\}$$
forms a basis for the vector space $\Gamma (\S^{n-1})$ over $\R$. We denote by  $\Gamma_{\le d} (\S^{n-1})$ the subspace of $\Gamma (\S^{n-1})$ consisting of the polynomial functions of degree at most $d$. It is also readily verified that the set of monomials 
 $$\B_{\le d} (n) := \left\{x_1^{i_1}\cdots x_n^{i_n} \in \B \mid \sum_{j = 1}^n i_j \le d \right\}$$
form a basis for the subspace $\Gamma_{\le d} (\S^{n-1})$. We denote by $\Gamma_d (\S^{n-1})$ the subspace generated by $\B_d  (n)$, where 
$$\B_d (n) := \left\{x_1^{i_1}\cdots x_n^{i_n} \in \B \mid \sum_{j=1}^n i_j = d \right\}.$$
Clearly, $\B_d(n)$ is a basis for the subspace $\Gamma_d (\S^{n-1})$. For any integers $n$ and $d$, we define 
$$M_{\le d} (n) := {n + d - 1 \choose d} + {n + d - 2 \choose d- 1},$$
{and} 
$$ M_{d} (n) := {n + d - 2 \choose d} + {n + d - 3 \choose d- 1}.$$
It follows from a simple counting argument that 
\begin{equation}\label{dim}
\dim_{\R} \Gamma_{\le d} (\S^{n-1}) = M_{\le d} (n) \ \ \text{and} \ \ \dim_{\R} \Gamma_{d} (\S^{n-1}) = M_{d} (n).
\end{equation}

\subsection{Canonical reduction modulo $I$} \label{red} Let $M = x_1^{i_1} \cdots x_n^{i_n} \in \R [x_1, \dots, x_n]$ be a monomial. Write $i_1 = 2 t + r$, where $0 \le r \le 1$. We define the \textit{the canonical reduction of $M$ modulo $I$}, denoted by $\bar{M}$, as follows: 
$$\bar{M} = (1 - x_2^2 - \dots - x_n^2)^t x_1^r x_2^{i_2} \cdots x_n^{i_n}.$$
That is, $\bar{M}$ is obtained by replacing every occurrence of $x_1^2$ by $1 - x_2^2 - \dots - x_n^2$. We readily note that $\bar{M}$ can be written uniquely as an $\R$-linear combination of elements of $\B$. This reduction is extended linearly to all polynomials $f (x_1, \dots, x_n) \in \R [x_1, \dots, x_n]$. We have the following observations: 
\begin{enumerate}
\item[(a)] $f \equiv \bar{f} (\text{mod} I)$.
\item[(b)] $f (P) = \bar{f} (P)$ for all $P \in \S^{n-1}$. 
\end{enumerate}
We note that, under the canonical reduction, the degree of a polynomial remains unchanged. However, if $M$ is a monomial of degree $d$, and $\bar{M} = \sum {a_i} M_i$, where $a_i \in \R$ and $M_i \in \B$, then either  $\deg M_i = d$ or $\deg M_i \le d-2$, whenever $a_i \neq 0$.

\section{Main theorem}\label{three}
\subsection{Proof of main theorem} In proving our main theorem stated below, we shall make use of the following well-known equality of binomial coefficients. For two positive integers $r < n$, we have
\begin{equation}\label{one}
{n - 1 \choose r-1} + {n - 1 \choose r} = {n \choose r}.
\end{equation}
We now state and prove our main result. 
\begin{thm}\label{main}
Let $\F \subset \S^{n-1}$ be a spherical $s$-distance set with $\P (\F) = \{t_1, \dots, t_s\}$. 
If $t_1 + \dots + t_s = 0$, then $$|\F| \le {n + s - 1 \choose s} + {n + s - 4 \choose s - 3}.$$ 
\end{thm}

\begin{proof}
Let $\F = \{v_1, \dots, v_m\}$. For each $i = 1, \dots, m$, we define 
$$f_i (x) = {\frac{\displaystyle{\prod_{k = 1}^s \left(\langle x, v_i \rangle - t_k \right)}}{\displaystyle{\prod_{k = 1}^s \left(1 - t_k \right)}}}.$$ 
It is obvious that $f_i (v_j) = \delta_{ij}$ where
$$\delta_{ij} =
\begin{cases}
 1 \ \ \ \text{if} \ \ \ i = j  \\
 0 \ \ \ \text{if} \ \ \ i \neq j.
\end{cases}$$
It follows from the observation made in Subsection \ref{red} that $\bar{f}_i \in \Gamma_{\le s} (\S^{n-1})$ for all $i = 1, \dots, m$ and that $\bar{f_i}(v_j) = \delta_{ij}$. Consequently, $\bar{f_1}, \dots, \bar{f_m}$ are linearly independent elements of $\Gamma_{\le s} (\S^{n-1})$ over $\R$. Let $\V$ denote the subspace of $\Gamma_{\le s} (\S^{n-1})$ spanned by $\bar{f_1}, \dots, \bar{f_m}$.

\textbf{Claim:} $\V \cap \Gamma_{s-1} (\S^{n-1}) =\{0\}$.

\textit{Proof of claim:} Let $c$ denote the constant $\prod_{k=1}^s (1 - t_k)$. Since $t_1 + \dots + t_s = 0$, each $cf_i$ is a linear combination of monomials of degree $s, s-2, s - 3, \dots, 0$. Furthermore, from the observations made in Subsection \ref{red}, we see that $\bar{cf_i}$ is also a linear combination of monomials of degree $s, s-2, s-3. \dots, 0$ for all $i = 1, \dots, m$.  The claim follows since no monomial of degree $s-1$ occur in the representation of $\bar{cf_i}$.

Since $\V$ and $\Gamma_{s-1}(\S^{n-1})$ are subspaces of $\Gamma_{\le s} (\S^{n-1})$, and they are disjoint, we see that
$\dim_\R \V + \dim_\R \Gamma_{s-1}(\S^{n-1}) \le \dim_\R \Gamma_{\le s} (\S^{n-1})$. Using equation \eqref{dim}, we derive
\begin{align*}
|\F| = m & \le {n+s - 1 \choose s} + {n + s - 2 \choose s-1} - {n + s - 3 \choose s-1} - {n+s-4 \choose s-2} \\
& = {n+s - 1 \choose s} + {n + s - 3 \choose s - 2} - {n + s - 4 \choose s - 2} \\
& = {n + s - 1 \choose s} + {n + s - 4 \choose s - 3}.
\end{align*}
The two equalities above follow from equation \eqref{one}. 
\end{proof}

\subsection{Comparison with other bounds} It follows trivially that our bound coincides with Gerzon's bound in the case when $s = 2$. We show that if $s$ is even, and $t_1 + \dots + t_s = 0$, then our bound is equal to the bound obtained by Barg and Musin. To this end, we need to prove the following combinatorial equality:

\begin{equation}
{n + s - 3 \choose s - 2} + \frac{n + 2s - 2}{s}{n + s - 3 \choose s - 1} = {n + s - 1 \choose s}
\end{equation}

To prove the above equality, we see that 
\begin{align*}
& \ {n + s - 1 \choose s} - {n + s - 3 \choose s - 2}  \\
&= {n + s - 1 \choose s} - {n + s - 2 \choose s - 1} + {n + s - 2 \choose s - 1} - {n + s - 3 \choose s - 2} \\
&= {n + s - 2 \choose s} + {n + s - 3 \choose s - 1} \\
&= \frac{(n + s - 2) !}{s! \ (n - 2) !} +  \frac{(n + s - 3) !}{(s - 1)! \ (n - 2) !} \\
&=\frac{(n + s - 3) !}{(s - 1)! \ (n - 2) !} \left(\frac{n + s - 2}{s} + 1 \right) \\
&= \frac{n + 2s - 2}{s} {n + s - 3 \choose s - 1}.
\end{align*}

This concludes the proof of our claim. It will indeed be interesting to investigate whether there exists a spherical $s$-distance set $\F$ with $\sum_{t \in \P(\F)} = 0$ that attains our bound.

\vspace{.5cm}

\end{document}